\documentclass[12pt]{article}
\usepackage{amsthm,amsmath,amssymb}
\usepackage[affil-it]{authblk}
\usepackage{proof}
\usepackage[cp1251]{inputenc}
\usepackage[english,russian]{babel}
\usepackage{fullpage}

\addto\captionsrussian{}

\newtheoremstyle{osn}{0.5\topsep}{0.5\topsep}{\slshape}{\parindent}%
{\bfseries}{.}{ }{}
\newtheoremstyle{vsp}{0.25\topsep}{0.25\topsep}{\upshape}{\parindent}%
{\bfseries}{.}{ }{}

\theoremstyle{osn}
\newtheorem{theorem}{Теорема}[section]

\newtheorem{lemma}[theorem]{Лемма}

\theoremstyle{vsp}

\newtheorem{example}[theorem]{Пример}

\newcommand{\df}{\nolinebreak\hspace{0pt}\mbox{-}\nolinebreak\hspace{0pt}\relax}
\newcommand{\ddf}{\nolinebreak\hspace{0pt}\mbox{---}\nolinebreak\hspace{0pt}\relax}

{\vspace{-0.4\baselineskip}
\begin{itemize}\setlength{\itemsep}{-\parsep}%
\addtolength{\itemsep}{0.1\baselineskip}}%
{\end{itemize}\vspace{-0.4\baselineskip}}

\newenvironment{menumerate}%
{\vspace{-0.4\baselineskip}
\begin{enumerate}\setlength{\itemsep}{-\parsep}%
\addtolength{\itemsep}{0.1\baselineskip}}%
{\end{enumerate}\vspace{-0.4\baselineskip}}

{\vspace{-0.4\baselineskip}
\begin{enumerate}\setlength{\itemsep}{-\parsep}%
\addtolength{\itemsep}{0.1\baselineskip}
}%
{\end{enumerate}\vspace{-0.4\baselineskip}}

{\vspace{-0.4\baselineskip}
\begin{description}\setlength{\itemsep}{-\parsep}
\addtolength{\itemsep}{0.1\baselineskip}}%
{\end{description}\vspace{-0.4\baselineskip}}

\newcommand{\twotrees}[2]{$$ \makebox[0.5\textwidth][l]{\qquad $ #1 $}%
\makebox[0.5\textwidth][l]{\qquad $ #2 $} $$}
\newcommand{\li}{\backslash}

\newcommand{\co}{\,}
\newcommand{\tp}{\mathrm{Tp}}
\newcommand{\lmbd}{\Lambda}
\newcommand{\arr}{\rightarrow}

\newcommand{\ldb}{\mathrm{L}^*(\li,!)}
\newcommand{\lmc}{\mathrm{LM}}
\newcommand{\lmsl}{\lmc^*(\li)}

\DeclareMathOperator{\metka}{label}
\DeclareMathOperator{\vid}{op}
\newcommand{\rnul}{\overset{\scriptscriptstyle 0}{g}}
\newcommand{\edu}{\overset{\scriptscriptstyle 2}{e}}
\newcommand{\cdu}{\overset{\scriptscriptstyle 2}{c}}
\newcommand{\fg}{f}
\newcommand{\dep}{\mathrm{d}}
\newcommand{\prs}{\mathrm{pr}_1(S)}
\newcommand{\sminus}{\diagdown}
\newcommand{\relo}{\mathrel{O}}

\newcommand{\os}{\mathrel{O^*}}
\newcommand{\op}{\mathrel{O^+}}
\newcommand{\hmap}{\mathfrak{h}}

\newcommand{\matirefm}[1]{\ref{#1}}
\newcommand{\uref}[1]{\textup{\ref{#1}}}

\allowdisplaybreaks

\begin{document}
%\large

%\runningtitle{Сложность исчисления $ \ldb $}
\title{Сети доказательства для исчисления Ламбека с~одним делением 
и модальностью для ослабления, используемой только при отрицательной полярности}

\author{Анна Евгеньевна Пентус}
\affil{Московский государственный университет
им.\ М.~В.~Ломоносова, механико-математический факультет}

\author{Мати Рейнович Пентус%
\thanks{Работа выполнена при поддержке РНФ (проект \mbox{18-11-00100}).}}
\affil{Московский государственный университет
им.\ М.~В.~Ломоносова, механико-математический факультет;\\
Российский государственный гуманитарный университет;\\
Санкт-Петербургский государственный университет}

\date{}
%\udk{512+510.64}
%\keywords{исчисление Ламбека}

\maketitle

\selectlanguage{english}

\begin{abstract}
Anna Pentus, Mati Pentus.
\textit{Proof nets for the Lambek calculus with one division and 
a~negative-polarity modality for weakening.}

\smallskip

In this paper, we introduce a~variant of the Lambek calculus allowing empty antecedents.
This variant uses two connecives: the left division and a~unary modality
that occurs only with negative polarity and allows weakening in antecedents of sequents.
We define the notion of a~proof net for this calculus,
which is similar to those for the ordinary Lambek calculus and 
multiplicative linear logic.
We prove that a~sequent is derivable in the calculus under consideration
if and only if there exists a~proof net for it.
Thus, we establish a~derivability criterion for this calculus
in terms of the existence of a~graph with certain properties.
The size of the graph is bounded by the length of the sequent.
\end{abstract}

\pagebreak[3]
\selectlanguage{russian}

\begin{abstract}
В~статье вводится 
вариант исчисления Ламбека, допускающего пустые антецеденты секвенций.
В этом варианте используются две связки: левое деление
и одноместная модальность, которая
встречается только с отрицательной полярностью и разрешает ослабление в антецеденте секвенции.
Определяется понятие сети доказательства для этого исчисления,
подобное аналогичным сетям для обычного исчисления Ламбека и линейной логики.
Доказывается, что произвольная заданная секвенция выводится в рассматриваемом исчислении 
тогда и только тогда, когда для неё существует сеть доказательства.
Тем самым устанавливается критерий для проверки выводимости в этом исчислении
в терминах существования графа с определёнными свойствами.
При этом размер графа ограничен длиной секвенции.
\end{abstract}

\section*{Введение}
Цель статьи "---
найти критерий выводимости в исчислении Ламбека
с~левым делением и модальностью ослабления, допускающем пустые антецеденты,
где модальность встречается только при отрицательной полярности.

В~разделе~\ref{s-Lambek-calculus} вводится исчисление $ \ldb $ "---
исчисление Ламбека
с~левым делением и модальностью ослабления, допускающее пустые антецеденты.
В~разделе~\ref{s-proof-nets}
доказывается критерий выводимости для исчисления $ \ldb $,
основанный на графах специального вида.

\section{Исчисление} \label{s-Lambek-calculus}
Исчисление Ламбека была введена в~\cite{Lam58}.
Рассмотрим исчисление Ламбека, допускающее пустые антецеденты,
обозначаемое $ \mathrm{L}^* $.
Добавим одноместную модальность ослабления, обозначаемую~$!$.
Рассотрим фрагмент, где разрешены только левое деление и модальность~$!$.
\emph{Типы} этого исчисления строятся из примитивных типов
$ p_1,p_2,\ldots $ с~помощью бинарной связки~$ \li $ и унарной связки~$!$.
Будем считать, что у связки~$!$ приоритет выше.
Другими словами, запись $ !A \li B $ означает
$ (!A) \li B $, а не
$ !(A \li B) $.

Обозначим множество всех типов, построенных таким образом, через~$ \tp(\li,!) $.
Прописные буквы латинского алфавита будем использовать для обозначения
типов, а~прописные греческие буквы "---
для обозначения конечных последовательностей типов.
Множество всех конечных последовательностей типов обозначается $ \tp(\li,!)^* $.
Пустая последовательность обозначается~$ \lmbd $.

Если $ \Gamma = A_1 \co\ldots\co A_n $,
то через $ !\Gamma $ обозначается последовательность 
$ (!A_1) \co\ldots\co (!A_n) $.
Через $ !^n A $ обозначим тип
$ \underbrace{ ! \ldots ! }_{n\text{ раз}} A $.

Выводимыми объектами
исчисления Ламбека с~левым делением и модальностью ослабления, допускающего пустые антецеденты,
являются \emph{секвенции} вида
$ \Gamma \arr A $.
(При этом $ \Gamma $ называется \emph{антецедентом},
а $ A $~--- \emph{сукцедентом} данной секвенции.)
Аксиомы имеют вид $ p_i \arr p_i $.
Выводы строятся с~помощью следующих правил:
\twotrees
{\infer[(\arr\li),]
{ \Pi \arr A \li B }
{ A \Pi \arr B }}
{\infer[(\li\arr),]
{ \Gamma \Pi ( A \li B ) \Delta \arr C }
{ \Pi \arr A
&
\Gamma B \Delta \arr C }}
\twotrees
{\infer[(!\arr),]
{ \Gamma !A \Delta \arr B }
{ \Gamma A \Delta \arr B }}
{\infer[(\mathrm{W}!\arr).]
{ \Gamma !A \Delta \arr B }
{ \Gamma \Delta \arr B }}

Рассматриваемое исчисление Ламбека
с~левым делением и модальностью ослабления, допускающее пустые антецеденты,
будем обозначать $ \ldb $.

\begin{example} \label{x-1li1li2-2}
Вывод
$$
\infer[(\li\arr)]
{ ( p_1 \li p_1 ) \li !p_2 \arr p_2 }
{\infer[(\arr\li)]
{ \arr p_1 \li p_1 }
{ p_1 \arr p_1 }
&
\infer[(!\arr)]
{!p_2 \arr p_2 }
{p_2 \arr p_2 }}
$$
показывает, что
$ \ldb \vdash ( p_1 \li p_1 ) \li !p_2 \arr p_2 $.
\end{example}

\begin{example} \label{x-2-1li1li2}
Вывод
$$
\infer[(\arr\li)]
{ p_2 \arr !( p_1 \li p_1 ) \li p_2 }
{\infer[(\mathrm{W}!\arr)]
{ !( p_1 \li p_1 ) \co p_2 \arr p_2 }
{ p_2 \arr p_2 }}
$$
показывает, что
$ \ldb \vdash p_2 \arr !( p_1 \li p_1 ) \li p_2 $.
\end{example}

Сформулируем критерий выводимости в~исчислении $ \ldb $,
аналогичный критериям для исчислений $ \mathrm{L}(\li) $ и $ \mathrm{L}^*(\li) $,
открытым Ю.~В.~Саватеевым (см.~\cite{Savateev2009Vestnik,Savateev2011}),
и критерию для исчисления $ \lmsl $
из~\cite{PP2012}.

\section{Сети} \label{s-proof-nets}
Сформулируем критерий выводимости для исчисления $ \ldb $,
основанный на графах специального вида.
Следуя традиции, принятой в линейной логике, такой граф будем называть
\emph{сетью доказательства}
(или, для краткости, просто \emph{сетью}).
В~конце этого раздела будет доказано, что секвенция выводится в~$ \ldb $
тогда и только тогда,
когда существует сеть, соответствующая данной секвенции.

Формально говоря, сеть доказательства
состоит из ориентированного дерева
$ \langle V , O \rangle $ 
с~множеством листьев~$ W $,
разметки листьев $ \metka \colon W \to \{ p_1, p_2,\ldots \} $,
разметки дуг $ \vid \colon O \to \{ -1 , 1 , 2 \} $,
строгого линейного порядка $ < $ на множестве~$ V $,
множества $ U \subseteq V $
и бинарного отношения
$ S \subseteq W \times W $.
Обозначим корень дерева
$ \langle V , O \rangle $
через~$ \rnul $.
Здесь дерево $ \langle V , O \rangle $
задаётся множеством вершин~$ V $
и множеством дуг~$ O $,
при этом дуги направлены в~сторону корня дерева,
то есть бинарное отношение~$ O $ является графиком функции
из множества $ V \sminus \{ \rnul \} $ в~множество~$ V $
(эту функцию будем также обозначать через~$ O $)
и для любого элемента
$ a \in V $
найдётся такое натуральное число~$ n $, что
$$
\underbrace{O(\ldots O}_{n\text{ раз}}(a)\ldots) = \rnul .
$$
Отметим, что функция $ \metka $ может не быть инъективной.

Первые четыре компоненты
(дерево $ \langle V , O \rangle $,
две разметки и линейный порядок~$ < $)
строятся по секвенции (однозначно) следующим образом.
Сеть для секвенции
$$
A_1 \co\ldots\co A_{n-1} \co A_n \!\to\! B
$$
совпадает с сетью для типа
$$
A_n \li ( A_{n-1} \li \ldots \li ( A_1 \li B ) \ldots ) ,
$$
а сети для типов строятся индукцией по построению типа.
Примитивному типу соответствует сеть с одной вершиной
(она помечена этим примитивным типом).

Сеть для типа $ !A $ получается из сети
для~$ A $ добавлением нового корня и дуги, ведущей
из корня сети для~$ A $ в~новый корень (у~этой дуги метка $2$).
Порядок тот же, что в~сети для~$ A $, 
но новый корень добавлен в~порядок как минимальный элемент.

Сеть для типа $ A \li B $ получается из дизъюнктного объединения сетей
для~$ A $ и~$ B $ добавлением нового корня и двух дуг, ведущих 
из корня сети для~$ A $ в~новый корень (у~этой дуги метка $-1$)
и из корня сети для~$ B $ в~новый корень (у~этогй дуги метка $1$).
Порядок определяется как сумма перевёрнутого порядка из сети для~$ A $,
тривиального порядка на новом корне (как одноэлементном множестве)
и (неперевёрнутого) порядка из сети для~$ B $.

Это завершает индуктивное определение размеченного дерева и линейного порядка 
для произвольного типа.

\smallskip

Число дуг с~меткой~$-1$ на пути из вешины~$ a $ до корня
обозначим через $\dep(a)$.
Например, $ \dep( \rnul ) = 0 $.

Через~$ \leq $ обозначим рефлексивное замыкание отношения~$ < $.
Через~$ \op $ обозначим транзитивное замыкание отношения~$ O $.
Через~$ \os $ обозначим рефлексивно\df транзитивное замыкание отношения~$ O $.

Для каждой вершины~$ a $ через $ H(a) $ обозначим тот (единственный) лист,
от которого к~вершине~$ a $ ведёт путь без отрицательных меток дуг.
Из определения~$ \dep $ очевидно, что
$ \dep(H(a)) = \dep(a) $.

\begin{lemma} \label{l-dep-0}
Если
$ \dep(a) = 0 $,
то
$ H(a) = H(\rnul) $.
\end{lemma}

Легко проверить (индукцией по построению типа), что в любой сети
истинны следующие утверждения.
\begin{menumerate}
\renewcommand{\theenumi}{D\arabic{enumi}}
\item \label{i-d1}
Если
$ a \leq b \leq c $,
$ a \os d $ и
$ c \os d $,
то
$ b \os d $.
\item \label{i-d-even}
Если
$ a \os b $ и
$ \dep(b) = 2n $,
то
$ a \leq H(b) $.
\item \label{i-d-odd}
Если
$ a \os b $ и
$ \dep(b) = 2n+1 $,
то
$ H(b) \leq a $.
\end{menumerate}

\begin{lemma} \label{l-h-even-odd}
\begin{menumerate}
\item \label{i-even}
Если
$ a \os b $ и
$ \dep(a) \neq \dep(b) = 2n $,
то
$ a < H(b) $.
\item \label{i-odd}
Если
$ a \os b $ и
$ \dep(a) \neq \dep(b) = 2n+1 $,
то
$ H(b) < a $.
\end{menumerate}
\end{lemma}

Обозначим через $ \prs $
первую проекцию бинарного отношения~$ S $.
Формально,
$$
 \prs = \{ a \mid \exists\, b \ (a,b) \in S \} .
$$

В~качестве последних двух компонент сети
можно взять любые множества
$ U \subseteq V $
и
$ S \subseteq W \times W $,
удовлетворяющие следующим условиям.

\begingroup
\makeatletter
\renewcommand{\@listI}{%
\leftmargin=50pt
\rightmargin=0pt
\labelsep=5pt
\labelwidth=45pt
\itemindent=0pt
\listparindent=0pt
\topsep=8pt plus 2pt minus 4pt
\partopsep=2pt plus 1pt minus 1pt
\parsep=0pt plus 1pt
\itemsep=\parsep}
\makeatother

\begin{menumerate}
\renewcommand{\theenumi}{PN\arabic{enumi}}
\item \label{i-pn-symm} \label{i-pn1}
Если
$ (a,b) \in S $,
то
$ (b,a) \in S $
(отношение~$ S $ симметрично).
\item \label{i-pn-func}
Если
$ (a,b) \in S $ и
$ (a,c) \in S $,
то
$ b = c $
(отношение~$ S $ функционально).
\item \label{i-pn-planar}
Если
$ (a,b) \in S $, $ (c,d) \in S $ и $ a < c < b $, то $ a < d < b $.
\item \label{i-pn-metka}
Если
$ (a,b) \in S $, то $ \metka( a ) = \metka( b ) $.
\item \label{i-pn-dep}
Если
$ (a,b) \in S $ и $ a < b $, то $ \dep( a ) = \dep( b ) + 1 $.
\item \label{i-pn-so}
Если
$ a \os c $, 
$ H(c) = b $,
$ a \in \prs $ 
и
$ a < S(b) < b $, 
то найдётся такой элемент~$ d $, что
$ S(a) \os d $
и
$ H(d) = b $.
\item \label{i-pn-rnul}
$ S \neq \varnothing $.
\item \label{i-pn-u-s}
$ \prs = W \cap U $.
\item \label{i-pn-up}
Если
$ a \relo b $ и $ b \notin U $, то $ a \notin U $.
\item \label{i-pn-down-op}
Если
$ a \relo b $, $ a \notin U $ и $ \vid(a,b) \neq 2 $, то $ b \notin U $.
\item \label{i-pn-down-dep}
%%Если
%%$ a \relo b $, $ a \notin U $ и $ \dep( b ) = 2n $, то $ b \notin U $.
Если
$ a \relo b $ и $ \dep( b ) = 2n $, то $ \vid(a,b) \neq 2 $.
\label{i-pnn}
\end{menumerate}
\endgroup

\begin{lemma} \label{l-h-s}
Если
$ S(a) $
определено, то
$ H(S(a)) = S(a) $.
\end{lemma}

\begin{lemma} \label{l-even-h-u}
Если
$ \dep(a) = 2n $
и
$ a \in U $,
то
$ H(a) \in U $.
\end{lemma}

\begin{proof}
%MM
Допустим, от противного, что 
$ H(a) \notin U $.
Согласно~\uref{i-pn-down-dep}
для всех вершин~$ b $
на пути от $ H(a) $
к~$ a $
имеем 
$ \dep( b ) = 2n $.
Согласно~\uref{i-pn-down-dep}
ни одна из этих вершин не принадлежит множеству~$ U $.
\end{proof}

\begin{lemma} \label{l-u}
Если
$ a \in U $
и 
и не существует такого элемента~$ b \in U \cap W $,
что
$ b \os a $,
то
найдётся такой элемент~$ c $,
что
$ O(c) = a $
и
$ \vid(c,a) = 2 $.
\end{lemma}

Согласно~\uref{i-pn-func} отношение~$ S $
является графиком частичной функции.
Эту частичную функцию будем также обозначать через~$ S $.

Для минимума, максимума, интервала, отрезка и
открытого справа полуинтервала
относительно линейного порядка~$ < $
будем использовать обозначения
$ \min $, $ \max $, $ ( a ; b ) $, $ [ a ; b ] $ и $ [ a ; b ) $.
Если
$ a \in V $ и $ \beta \subseteq V $,
то запись $ a < \beta $ означает, что для всех $ b \in \beta $
имеет место $ a < b $.

На диаграммах будем располагать вершины дерева на плоскости так, чтобы
движение слева направо соответствовало линейному порядку~$ < $
(если $ a < b $, то $ a $ левее, чем $ b $),
а дуги дерева вели бы снизу вверх.
На диаграмме лист~$ a $ представляется примитивным типом $ \metka( a ) $
(с~чёрточкой сверху, если число $ \dep( a ) $ нечётное).
Бинарное отношение~$ O $ изображается прямыми стрелками,
а отношение~$ S $ "---
трёхзвенными ломаными линиями.

\begin{lemma} \label{l-arr-li}
Каждая сеть для секвенции
$ \Pi \arr A \li B $
является сетью также для секвенции
$ A \co \Pi \arr B $,
и наоборот.
\end{lemma}

\begin{theorem} \label{t-pn-sound}
Пусть
$ \ldb \vdash \Phi \arr D $.
Тогда существует сеть, соответствующая секвенции $ \Phi \arr D $.
\end{theorem}

\begin{proof}
Индукция по длине вывода.
Рассмотрим в шаге индукции случай правила
\twotrees
{\infer[(\li\arr).]
{ \Gamma \Pi ( A \li B ) \Delta \arr C }
{ \Pi \arr A
&
\Gamma B \Delta \arr C }}
{}
Пусть для секвенции
$ \Pi \arr A $
предположение индукции
даёт множество~$ U' $ и бинарное отношение~$ S' $,
а для секвенции
$\Gamma B \Delta \arr C $
множество~$ U'' $ и бинарное отношение~$ S'' $.
Если корень поддерева, соответствующего типу~$ B $,
принадлежит
$ U'' $,
то
зададим множество~$ U $ и отношение~$ S $
для $ \Gamma \Pi ( A \li B ) \Delta \arr C $
как естественные образы множества~$ U'' $ и отношения~$ S'' $,
иначе 
зададим множество~$ U $ 
как объединение естественных образов
множеств $ U' $ и~$ U'' $,
а отношение~$ S $
как объединение естественных образов
отношений $ S' $ и~$ S'' $.

Остальные случаи в~шаге индукции доказываются легко.
\end{proof}

\begin{theorem} \label{t-pn-complete}
Пусть
существует сеть, соответствующая секвенции $ \Phi \arr D $.
Тогда
$ \ldb \vdash \Phi \arr D $.
\end{theorem}

Теорема~\ref{t-pn-complete} доказывается индукцией по
суммарному количеству связок в~$ \Phi $ и~$ D $.
Оставшаяся часть 
%этого раздела 
статьи
посвящена доказательству шага индукции.
Согласно лемме~\matirefm{l-arr-li}
достаточно ограничиться секвенциями, где сукцедент $ D $
является примитивным типом.
Рассмотрим произвольную секвенцию $ \Phi \arr p_i $
и произвольную соответствующую ей сеть
$ \langle V , O , \metka , \vid , < , U , S \rangle $
и предположим, что для всех секвенций с меньшим количеством связок
теорема уже доказана.

\begin{lemma} \label{l-less-rnul}
Если
$ \dep(a) > 0 $,
то
$ a < H(\rnul) $.
\end{lemma}

\begin{proof}
В силу леммы~\ref{l-h-even-odd}
из
$ a \os \rnul $
и
$ \dep(a) > 0 = \dep(\rnul) $
получаем
$ a < H(\rnul) $.
\end{proof}

\begin{lemma} \label{l-pn-rnul}
$ H(\rnul) \in \prs $.
\end{lemma}

\begin{proof}
Допустим, от противного, что 
$ H(\rnul) \notin \prs $.
Очевидно,
$ H(\rnul) \notin W $.
Согласно~\uref{i-pn-u-s}
$ H(\rnul) \notin U $.
Согласно лемме~\uref{l-even-h-u}
получаем
$ \rnul \notin U $.
Из~\uref{i-pn-up}
индукцией по длине пути получаем, что ни один элемент множества~$ V $
не принадлежит множеству~$ U $.
Согласно~\uref{i-pn-u-s}
$ \prs = \varnothing $.
Это противоречит
\uref{i-pn-rnul}.
\end{proof}

\begin{lemma} \label{l-dep-rnul}
$ \dep(S(H(\rnul))) = 1 $.
\end{lemma}

\begin{lemma} \label{l-pn-basic}
%%M: a -> S(b)
Если
$ a \op c $,
$ \dep(a) > \dep( c ) $
и
$ a < H( c ) < S( a ) $,
то 
$ H( c ) \in \prs $
и
$ H( c ) < S(H( c )) < S( a ) $.
\end{lemma}

\begin{proof}
Если
$ \dep( c ) $ нечётно,
то согласно~\uref{i-d-odd}
из $ a \os c $
получаем $ H( c ) \leq a $,
то есть $ H( c ) \leq a $,
что противоречит условию
$ a < H( c ) $.
Следовательно,
$ \dep( c ) $ чётно.
Согласно~\uref{i-d-even}
из $ c \os c $
получаем $ c \leq H( c ) $.

Согласно~\uref{i-pn-up}
из $ a \os c $ и $ a \in U $
следует
$ c \in U $.
Согласно лемме~\uref{l-even-h-u}
$ H( c ) \in U $, 
откуда получаем
$ H( c ) \in \prs $.
Обозначим
$ S(H( c )) = S( H( c ) ) $.
Согласно~\uref{i-pn-planar}
имеем, что
$ a < S(H( c )) < S( a ) $.
Если
$ H( c ) < S(H( c )) $,
то лемма доказана.
В~силу~\uref{i-pn-dep}
случай
$ S(H( c )) = H( c ) $
невозможен.
Осталось рассмотреть случай
$ S(H( c )) < H( c ) $.

Из
$ a \os c $, 
$ a \in \prs $ 
и
$ a < S(H( c )) < H( c ) $, 
согласно~\uref{i-pn-so}
получаем, что
найдётся такой элемент~$ d $, что
$ S(a) \os d $
и
$ H(d) = H( c ) $.
Так как
$ H(d) = H( c ) $,
то
$ \dep(d) = \dep( c ) $.
Следовательно,
$ \dep(d) $ чётно.
Согласно~\uref{i-d-even}
из $ S(a) \os d $
получаем $ S(a) \leq H(d) = H( c ) $,
что противоречит неравенству
$ H( c ) < S(a) $.
\end{proof}

Обозначим
$$
\delta = \{ a \mid \dep(a) \!=\! 2 \land a \!\in\! U \land
 \vid(a,O(a)) \!=\! -1 \land ( H(O(a)) \!=\! S(H(\rnul)) \lor H(O(a)) \!\notin\! U ) \} .
$$

\begin{lemma} \label{l-Theta-p-Xi-p}
Пусть
множество~$ \delta $
пусто.
Тогда исходная секвенция $ \Phi \arr p_i $
имеет вид
$ !\Theta \co !^n p_i \co !\Xi \arr p_i $.
\end{lemma}

\begin{proof}
Очевидно,
$ \metka(H(\rnul)) = p_i $.
Согласно условиям~\uref{i-pn-metka} и~\uref{i-pn-dep}
имеем, что
$ \metka\bigl(S(H(\rnul))\bigr) = p_i $
и
$ \dep\bigl(S(H(\rnul))\bigr) = 1 $.
Обозначим через~$ D $ тип из антецедента, содержащий вершину~$ S(H(\rnul)) $.
Рассмотрим случай, когда $ D $ содержит деление и, следовательно, имеет вид $ !^n ( A \li B ) $.
Обозначим через~$ a $ корень поддерева, соответствующего типу~$ A $.
Тогда $ O(a) $ является корнем поддерева, соответствующего типу~$ A \li B $.
Очевидно,
$ \dep(O(a)) = 1 $
и
$ \vid(a,O(a)) = -1 $,
откуда получаем
$ \dep(a) = 2 $.
Так как
$ \dep\bigl(H(O(a))\bigr) = \dep\bigl(O(a)\bigr) = 1 $
и в~поддереве, соответствующем типу~$ A \li B $,
только один лист с~единицей в~качестве значения функции~$\dep$,
то
$ H(O(a)) = S(H(\rnul)) $.
Поэтому
$ a \in \delta $,
что противоречит условиям леммы.

Осталось рассмотреть случай, когда $ D $ не содержит деления 
и, следовательно, имеет вид $ !^n p_i $.
Единственным листом в~$ D $
является
$ S(H(\rnul)) $.
Если 
$ U \cap W = \{ H(\rnul) , S(H(\rnul)) \} $,
то согласно лемме~\ref{l-u}
все типы антецедента, кроме~$ D $, начинаются с~$!$
и, следовательно, лемма доказана.
Иначе обозначим
$ b = \min \bigl( ( U \cap W ) \sminus \{ H(\rnul) , S(H(\rnul)) \} \bigr) $.
Обозначим через~$ c $
элемент, удовлетворяющий условиям
$ b \os c $
и
$ H(O(c)) = H(\rnul) $.
%самый правый из тех типов антецедента,
%чьё поддерево содержит элемент множества
%$ ( U \cap W ) \sminus \{ S(H(\rnul)) \} $.
%Обозначим корень поддерева, соответствующего этому типу, через $ c $.
В~силу~\ref{i-pn-up},
из
$ b \in U $
получаем
$ c \in U $.

Рассмотрим сначала случай, когда
$ H(c) \in U $.
Очевидно,
$ \dep(H(c)) = \dep(c) = 1 $.
Согласно~\ref{i-pn-u-s}
определено значение
$ S(H(c)) $.
Если
$ S(H(c)) < H(c) $,
то
невозможно $ S(H(c)) \os c $
и поэтому $ S(H(c)) $ принадлежит поддереву какого-то другого типа
и этот тип не может быть~$ D $,
а это противоречит выбору~$ c $.
Пусть
$ H(c) < S(H(c)) $.
Тогда
$ \dep(S(H(c))) = 0 $.
Согласно леммам~\ref{l-h-s}
$ S(H(c)) = H(S(H(c))) = H(\rnul) $.
В~силу симметричности~$S$ имеем
$ H(c) = S(H(\rnul)) $.
Тогда
$ H(c) $
принадлежат поддереву, соответствующему типу~$ D $,
что противоречит выбору~$ c $.

Теперь рассмотрим случай, когда
$ H(c) \notin U $.
Тогда тип антецедента, соответствующий вершине~$ c $,
имеет вид $ A \li B $.
Обозначим корень поддерева, соответствующего типу~$ A $, через $ a $.
Из
$ \dep(c) = 1 $
получаем
$ \dep(a) = 2 $.
Согласно~\ref{i-pn-down-op}
имеем
$ a \in U $.
Из
$ H(O(a)) = H(c) \notin U $
и
$ \vid(a,O(a)) = -1 $
получаем
$ a \in \delta $,
что противоречит условиям леммы.
\end{proof}

Осталось рассмотреть случай, когда
$ \delta \neq \varnothing $.
Обозначим
$$
\edu = \max \{ a \mid
 O( a ) = \min \{ O( b ) \mid b \in \delta \}
\} .
$$

В~следующих леммах (до конца раздела~\ref{s-proof-nets})
предполагается, что
$ \delta \neq \varnothing $.

\begin{lemma} \label{l-dep-edu}
$ \dep( \edu ) = 2 $.
\end{lemma}

\begin{proof}
Очевидно из определения~$ \delta $.
\end{proof}

Обозначим
$$
\alpha = \{ a \mid a \os \edu \} ,\quad
\varphi = \{ b \in \prs \mid H(\edu) < b \land ( S( b ) = H(\rnul) \lor S( b ) < \alpha ) \} .
$$
%MM: f \li p, p \li !z, e, e \li !z, q, (f\q)\s \arr s is not an error. Here \edu = p.
%MM: f \li !z, e, e \li !z, q, (f\q)\s \arr s  is OK
%MM: f \li !z, q, (f\q)\s \arr s  is OK

\begin{lemma} \label{l-Srnul}
Пусть
$ S(H(\rnul)) < \edu $
и
$ a \in \alpha $.
Тогда
$ S(H(\rnul)) < a $.
\end{lemma}

\begin{proof}
Допустим, от противного, что 
$ a \leq S(H(\rnul)) $.
Условие
$ a \in \alpha $,
означает, что
$ a \os \edu $.
%Из~\ref{i-d-even}
%получаем
%$ \edu \leq H(\edu) $.
Так как
$ a \leq S(H(\rnul)) \leq \edu $,
из~\ref{i-d1}
получаем
$ S(H(\rnul)) \os \edu $.
Это противоречит неравенству
$ \dep(S(H(\rnul))) = 1 < 2 = \dep(\edu) $.
\end{proof}

\begin{lemma} \label{l-phi}
Множество~$ \varphi $
непусто.
\end{lemma}

\begin{proof}
Если 
$ \edu < S(H(\rnul)) $,
то
$ S(H(\rnul)) \in \varphi $.

Случай
$ \edu = S(H(\rnul)) $
невозможен, так как
$ \dep(\edu) = 2 \neq 1 = \dep(S(H(\rnul))) $.

Пусть теперь
$ S(H(\rnul)) < \edu $.
Докажем, что
$ H(\rnul) \in \varphi $.
Утверждение $ H(\rnul) \in \prs $
доказано в лемме~\ref{l-pn-rnul}.
Из лемм~\ref{l-dep-edu} и~\ref{l-less-rnul}
получаем, что
$ \edu < H(\rnul) $.
Согласно лемме~\ref{l-Srnul} имеем, что
$ S(H(\rnul)) < \alpha $.
\end{proof}

Обозначим
$$
\fg = \min \varphi ,\quad
\pi = ( H(\edu) ; \fg ) .
$$

\begin{lemma} \label{l-dep-fg}
$ \dep( \fg ) \leq 1 $.
\end{lemma}

\begin{proof}
По построению множества~$ \varphi $
из
$ \fg \in \varphi $
следует
$ S( \fg ) = H(\rnul) $ 
или 
$ S( \fg ) < \alpha $.
В~первом случае
$ \dep( \fg ) = 1 $
в~силу леммы~\ref{l-dep-rnul}.

Осталось рассмотреть случай
$ S( \fg ) < \alpha $.
Тогда из
$ S( \fg ) < H(\edu) $
и
$ H(\edu) < \fg $
получаем
$ S( \fg ) < \fg $.

Допустим, от противного, что
$ \dep(\fg) \geq 2 $.
Согласно~\uref{i-pn-dep}
имеем, что
$ \dep\bigl(S(\fg)\bigr) \geq 3 $.
Обозначим через~$ \cdu $
элемент, удовлетворяющий условиям
$ S(\fg) \op \cdu $,
$ \dep( \cdu ) = 2 $
и
$ \dep(O( \cdu )) = 1 $.
Из того, что $ S(\fg) < H(\edu) $, $ S(\fg) \op \cdu $
и $ \dep( \cdu ) = \dep(H(\edu)) $
%MM
получаем, что
$ \cdu \leq \edu $.
Из того, что
$ S(\fg) \notin \alpha $,
получаем, что
$ \cdu \neq \edu $.
Итого,
$ \cdu < \edu $,
и даже
%MM
$ H( \cdu ) < \alpha $.

%Из того, что $ S(\fg) < \fg < H(\rnul) $, 
%согласно~\uref{i-pn-planar} получаем, что $ S(H(\rnul)) < S(\fg) $.
Согласно лемме~\uref{l-h-even-odd} имеем $ S(\fg) < \cdu $.
Согласно лемме~\ref{l-pn-basic} имеем
$ H( \cdu ) < S(H( \cdu )) $.
%Итого, $ S(H(\rnul)) < S(H( \cdu )) $.
%Из того, что $ S(H(\rnul)) < S(H( \cdu )) $, $ H(\edu) \op S(H(\rnul)) $
%и $ \dep\bigl(S(H(\rnul))\bigr) = \dep\bigl(S(H( \cdu ))\bigr) $,
Из~\uref{i-d1} получаем, что
$ H(\edu) < S(H( \cdu )) $.

Мы доказали, что
$ S(H( \cdu )) \in \varphi $.
Следовательно,
$ \fg \leq S(H( \cdu )) $.
Имеем
$$
 S(\fg) < H( \cdu ) < \fg \leq S(H( \cdu )) ,
$$
что противоречит~\uref{i-pn-planar}.
\end{proof}

\begin{lemma} \label{l-alpha-pi-fg}
Пусть
$ a \in \alpha \cup \pi $.
Тогда
$ a < \fg $.
\end{lemma}

\begin{proof}
Случай
$ a \in \pi $
очевиден.
Пусть
$ a \in \alpha $,
т.~е.\
$ a \os \edu $.
По лемме~\matirefm{l-dep-edu}
согласно~\uref{i-d-even}
получаем, что
$ a \leq H(\edu) $.
Так как
$ \fg \in \varphi $,
то
$ H(\edu) < \fg $,
что влечёт
$ a < \fg $.
\end{proof}

\begin{lemma} \label{l-alpha-pi-interval}
Пусть
$ a \leq b \leq c $,
$ a \in \alpha \cup \pi $
и
$ c \in \alpha \cup \pi $.
Тогда
$ b \in \alpha \cup \pi $.
\end{lemma}

\begin{proof}
Сначала рассмотрим случай
$ b \leq H(\edu) $.
Так как
$ a \notin \pi $,
имеем
$ a \in \alpha $.
По~\uref{i-d1}
из
$ a \os \edu $
и
$ H(\edu) \os \edu $
получаем, что
$ b \os \edu $,
т.~е.\
$ b \in \alpha $.

Теперь рассмотрим случай
$ H(\edu) < b $.
Согласно лемме~\matirefm{l-alpha-pi-fg}
имеем
$ c < \fg $.
Получаем, что
$ H(\edu) < b < \fg $,
т.~е.\
$ b \in \pi $.
\end{proof}

\begin{lemma} \label{l-alpha-pi-closed}
Пусть
$ ( a , b ) \in S $
и
$ a \in \alpha \cup \pi $.
Тогда
$ b \in \alpha \cup \pi $.
\end{lemma}

\begin{proof}
Сначала рассмотрим случай
$ a < b $.
Если
$ b < \fg $,
то из того, что
$ a < b \leq \max [ H(\edu) ; \fg ) $,
согласно лемме~\ref{l-alpha-pi-interval}
получаем, что
$ b \in \alpha \cup \pi $.
Пусть
$ \fg \leq b $.
Согласно~\uref{i-pn-planar}
имеем
$ a \leq S(\fg) < b $.
Так как
$ \fg \in \varphi $,
имеем
$ S(\fg) < H(\edu) $.
Следовательно,
$ a < H(\edu) $,
что влечёт
$ a \notin \pi $
и
$ a \in \alpha $,
т.~е.\
$ a \os \edu $.
По~\uref{i-d1}
из
$ a \leq S(\fg) < H(\edu) $
получаем, что
$ S(\fg) \os \edu $,
т.~е.\
$ S(\fg) \in \alpha $.
С~другой стороны,
$ \fg \in \varphi $
влечёт
$ S(\fg) < \alpha $,
противоречие.

Теперь рассмотрим случай
$ b < a $.
Допустим, от противного, что
$ b \notin \alpha \cup \pi $.
Согласно лемме~\matirefm{l-alpha-pi-interval}
имеем, что
$ b < \alpha $.

Если
$ H(\edu) < a $,
то
$ a \in \varphi $,
что влечёт
$ \fg \leq a $,
а это противоречит лемме~\matirefm{l-alpha-pi-fg}.
Осталось рассмотреть случай
$ a \leq H(\edu) $.
Очевидно,
$ a \in \alpha $,
т.~е.\
$ a \os \edu $.
Согласно лемме~\matirefm{l-dep-edu}
имеем, что
$ \dep( a ) \geq 2 $.
По~\uref{i-pn-dep}
получаем, что
$ \dep( b ) \geq 3 $.
Обозначим через~$ c $
элемент, удовлетворяющий условиям
$ b \op c $
и
$ \dep( c ) = 2 $.
Из
$ b < H(\edu) $
по~\uref{i-d1}
получаем, что
$ H(c) \leq H(\edu) $.
Из
$ b \notin \alpha $
следует, что
$ H(c) \neq H(\edu) $.
Из
$ H(c) < H(\edu) $
и
$ \dep( H(c) ) = \dep( \edu ) $
получаем, что
$ H(c) < \alpha $
и, в~частности,
$ H(c) < a $.
Из~\uref{i-d-even}
получаем, что
$ b \leq H(c) $.
Из того, что
$ b \leq H(c) < a = S( b ) $
и
$ b \op c $,
по лемме~\matirefm{l-pn-basic}
получаем, что
$ H(c) \in \prs $
и
$ H(c) < S( H(c) ) < a $.
По~\uref{i-pn-dep} имеем
$ \dep\bigl(S( H(c) )\bigr) = 1 $.
Из леммы~\ref{l-Srnul}
по~\uref{i-pn-planar}
получаем, что
$ S(H(\rnul)) < b $.
Следовательно,
$ S(H(\rnul)) < S( H(c) ) $.
Из того, что
$ S(H(\rnul)) < S( H(c) ) < H(\edu) $,
согласно~\uref{i-d1}
получаем, что
$ S( H(c) ) \os S(H(\rnul)) $.
Так как
$ \dep\bigl(S( H(c) )\bigr) = \dep\bigl(S(H(\rnul))\bigr) $,
получаем, что
$ S( H(c) ) = S(H(\rnul)) $,
противоречие.
\end{proof}

\begin{lemma} \label{l-pi}
Пусть
$ b \op a $
и
$ \dep( a ) = 1 $.
Тогда условия
$ a \in \pi $
и
$ b \in \pi $
равносильны.
\end{lemma}

\begin{proof}
В~силу леммы~\uref{l-h-even-odd}
имеем, что
%MM
$ a < b $.

Сначала допустим, от противного, что
$ a \in \pi $,
но
$ b \notin \pi $.
Очевидно,
$ \fg \leq b $.
Согласно условию леммы
$ \dep(b) \neq 0 $,
т.~е.\
$ b \neq H(\rnul) $.
Следовательно,
$ b < H(\rnul) $.
Получаем, что
$ f \neq H(\rnul) $,
т.~е.\
$ \dep(\fg) \neq 0 $.
Согласно лемме~\matirefm{l-dep-fg}
имеем
$ \dep(\fg) = 1 $.
Из
$ \fg \leq b $,
$ b \op a $
и
$ \dep(\fg) = \dep( a ) $
согласно~\uref{i-d1}
выводим, что $ \fg \leq a $,
что противоречит условию
$ a \in \pi $.

Теперь допустим, от противного, что
$ b \in \pi $,
но
$ a \notin \pi $.
Очевидно,
$ a \leq H(\edu) < b $.
По~\uref{i-d1}
получаем, что
$ H(\edu) \os a $.
Из леммы~\ref{l-dep-edu}
и того, что
$ \dep( a ) = 1 $,
получаем, что
$ a = O(\edu) = S(H(\rnul)) $.
Обозначим через~$ c $
элемент, удовлетворяющий условиям
$ b \os c $
и
$ O( c ) = a $.
Очевидно,
$ c \in \delta $.
Следовательно,
$ c \leq H(\edu) < b $.
Так как
$ \dep( c ) = 2 $,
то
$ c < b $
противоречит~\uref{i-d-even}.
\end{proof}

\begin{lemma} \label{l-long}
Ограничение сети
$ \langle V , O , \metka , \vid , < , U , S \rangle $
на множество
$ V \sminus ( \alpha \cup \pi ) $
является сетью для некоторой секвенции.
\end{lemma}

\begin{proof}
Согласно построению
вершина~$ S(H(\rnul)) $
является корнем некоторого поддерева, соответствующего некоторому
типу вида~$ A \li B $
из антецедента.
При этом
множество~$ \alpha $
соответствует типу~$ A $.
Согласно лемме~\matirefm{l-pi}
исходная секвенция имеет вид
$ \Gamma \co \Pi \co ( A \li B ) \co \Delta \arr p_i $,
где
множество~$ \pi $
соответствует
последовательности~$ \Pi $.
Следовательно, множество
$ V \sminus ( \alpha \cup \pi ) $
соответствует секвенции
$ \Gamma \co B \co \Delta \arr p_i $.
Для ограничения сети
$ \langle V , O , \metka , \vid , < , U , S \rangle $
на множество
$ V \sminus ( \alpha \cup \pi ) $
легко проверить условия \uref{i-pn1}\ddf\uref{i-pnn}.
%Например, при проверке~\uref{i-pn-up} и~\uref{i-pn-down}
%используется лемма~\matirefm{l-alpha-pi-closed}.
\end{proof}

\begin{lemma} \label{l-step}
Найдутся такие типы~$ A $ и~$ B $ и последовательности~$ \Gamma $ и~$ \Pi $,
что
$$
 \Phi = \Gamma \co \Pi \co ( A \li B ) \co \Delta 
$$
и существуют сети для секвенций
$ \Pi \arr A $
и
$ \Gamma \co B \co \Delta \arr p_i $.
\end{lemma}

\begin{proof}
Определим
$ A $ и~$ \Pi $
как в~доказательстве леммы~\matirefm{l-long},
где предъявлена сеть для секвенции
$ \Gamma \co B \co \Delta \arr p_i $.
Сеть для секвенции
$ \Pi \arr A $
получается из
ограничения сети
$ \langle V , O , \metka , \vid , < , U , S \rangle $
на множество
$ \alpha \cup \pi $
очевидными преобразованиями:
перестановкой двух частей ($ \alpha $ и $ \pi $)
в~смысле отношения~$ < $
и объявлением вершины~$ \edu $ корнем всего дерева,
что приводит к~уменьшению глубины каждой вершины из~$ \alpha $ на~$ 2 $.

Формально говоря, пусть
секвенции
$ \Pi \arr A $
соответствуют
дерево $ \langle V' , O' \rangle $,
разметка $ \metka' $ и линейный порядок~$ <' $.
Глубину вершины~$a$
в~дереве $ \langle V' , O' \rangle $
обозначим через $\dep'(a)$.
Легко доказать, что существует биекция
$ \hmap \colon V' \to \alpha \cup \pi $
со следующими свойствами:
\begin{align*}
& a \relo' b \Leftrightarrow
( \hmap(a) \relo \hmap(b) \land \hmap(a) \in \alpha \land \hmap(b) \in \alpha )
\lor{}
\\
&\quad{}\lor
( \hmap(a) \relo \hmap(b) \land \hmap(a) \in \pi \land \hmap(b) \in \pi )
\lor
( \hmap(a) \relo H(\rnul) \land \hmap(b) = \edu ) ,
\\
& \metka'( a ) = \metka\bigl(\hmap(a)\bigr) ,
\\
& a <' b \Leftrightarrow
( \hmap(a) < \hmap(b) \land \hmap(a) \in \alpha \land \hmap(b) \in \alpha )
\lor{}
\\
&\quad{}\lor
( \hmap(a) < \hmap(b) \land \hmap(a) \in \pi \land \hmap(b) \in \pi )
\lor
( \hmap(a) \in \pi \land \hmap(b) \in \alpha ) ,
\\
& \dep'( a ) =
\begin{cases}
\dep\bigl(\hmap(a)\bigr) - 2, &\text{если } \hmap(a) \in \alpha ,\\
\dep\bigl(\hmap(a)\bigr), &\text{если } \hmap(a) \in \pi .
\end{cases}
\end{align*}

Искомая сеть для секвенции
$ \Pi \arr A $
задаётся 
множеством
$$
U' = \{ a \mid \hmap(a) \in U \} 
$$
и
бинарным отношением
$$
S' = \bigl\{ ( a , b ) \mid \bigl(\hmap(a) , \hmap(b)\bigr) \in S \bigr\} .
$$
Осталось проверить условия~\uref{i-pn1}---\uref{i-pnn}
для~$ S' $.

Докажем~\uref{i-pn-up}
для~$ U' $.
Пусть
$ a \in U' $.
Если
$ \hmap(a) \in \alpha $
и
$ \hmap\bigl(O'(a)\bigr) \in \alpha $,
то
$ O'(a) \in U' $
следует из условия~\uref{i-pn-up} для~$ S $.
Аналогично если
$ \hmap(a) \in \pi $
и
$ \hmap\bigl(O'(a)\bigr) \in \pi $,
то
$ O'(a) \in U' $
следует из условия~\uref{i-pn-up}
для~$ S $.
Если же
$ \hmap(a) \in \pi $
и
$ \hmap\bigl(O'(a)\bigr) \in \alpha $,
то
$ \hmap\bigl(O'(a)\bigr) = \edu \in \prs $,
что влечёт
$ O'(a) \in U' $.

Условие~\uref{i-pn-rnul},
т.~е.\
$ S' \neq \varnothing $,
следует из того, что
$ \edu \in \alpha $
и
$ \edu \in \prs $.

Остальные условия из определения сети проверяются легко.
\end{proof}

Для завершения доказательства теоремы~\ref{t-pn-complete}
достаточно заметить, что
шаг индукции
следует из лемм~\matirefm{l-Theta-p-Xi-p} и~\matirefm{l-step}.
В~первом случае исходная секвенция $ \Phi \arr D $
выводится из аксиомы многократным применением 
правил $ (!\arr) $ и $ (\mathrm{W}!\arr) $,
во~втором случае секвенция $ \Phi \arr D $
получается по правилу~$ (\li\arr) $
из двух более коротких секвенций, к~которым применимо предположение индукции.
Теорема~\ref{t-pn-complete} доказана.

\end{document}